\documentclass[12pt]{amsart}
\usepackage{amsmath,amsthm,amsfonts,amssymb,mathrsfs}
\date{\today}

\usepackage[cp1251]{inputenc}         

\usepackage{amsmath,amsthm,amsfonts,amssymb,mathrsfs}
\usepackage{eucal}

\usepackage{color}

\usepackage{amssymb,cite}
\usepackage[colorlinks,plainpages,citecolor=magenta, linkcolor=blue, backref]{hyperref}

\usepackage{hyperref}

\newtheorem{theorem}{Theorem}
\newtheorem{lemma}{Lemma}

\newtheorem{proposition}{Proposition}

\theoremstyle{definition}

\newtheorem{example}{Example}

\theoremstyle{remark}
\newtheorem{remark}{Remark}

\numberwithin{equation}{section}

\begin{document}

\title[On topologization of the bicyclic monoid]{On topologization of the bicyclic monoid}

\author[A. Chornenka and O.~Gutik]{Adriana Chornenka and Oleg~Gutik}
\address{Ivan Franko National University of Lviv, Universytetska 1, Lviv, 79000, Ukraine}
\email{adriana.chornenka@lnu.edu.ua, oleg.gutik@lnu.edu.ua, ogutik@gmail.com}


\begin{abstract}
We construct two non-discrete inverse semigroup $T_1$-topologies and a compact inverse shift-continuous $T_1$-topology on the bicyclic monoid ${\mathscr{C}}(p,q)$. Also we give  conditions on a $T_1$-topology $\tau$ on ${\mathscr{C}}(p,q)$ to be discrete. In particular, we show that if $\tau$ is an inverse semigroup $T_1$-topology on ${\mathscr{C}}(p,q)$ which satisfies one of the following conditions: $\tau$ is Baire, $\tau$ is quasi-regular or $\tau$ is semiregular, then $\tau$ is discrete.
\end{abstract}

\maketitle


\section{\textbf{Introduction and preliminaries}}

In this paper we shall follow the terminology of \cite{Carruth-Hildebrant-Koch=1983, Carruth-Hildebrant-Koch=1986, Clifford-Preston=1961, Clifford-Preston=1967, Engelking=1989, Lawson=1998, Ruppert=1984}.

\smallskip

If $(X,\tau)$ is a topological space and $Y\subseteq X$, then we mean that  $Y$ is a subspace of $(X,\tau)$ and by $\operatorname{cl}_Y(A)$ and $\operatorname{int}_Y(A)$ we denote the closure and the interior, respectively, of  $A\subseteq Y$ in the topological space $Y$.

\smallskip

A semigroup $S$ is called {\it inverse} if for any
element $x\in S$ there exists a unique $x^{-1}\in S$ such that
$xx^{-1}x=x$ and $x^{-1}xx^{-1}=x^{-1}$. The element $x^{-1}$ is
called the {\it inverse of} $x\in S$. If $S$ is an inverse
semigroup, then the function $\operatorname{inv}\colon S\to S$
which assigns to every element $x$ of $S$ its inverse element
$x^{-1}$ is called the {\it inversion}.
On an inverse semigroup $S$ the semigroup operation  determines the following partial order $\preccurlyeq$: $s\preccurlyeq t$ if and only if there exists $e\in E(S)$ such that $s=te$. This partial order is called the \emph{natural partial order} on $S$.

\smallskip

A (\emph{semi})\emph{topological} \emph{semigroup} is a topological space with a (separately) continuous semigroup operation. An inverse topological semigroup with continuous inversion is called a \emph{topological inverse semigroup}.

\smallskip

A topology $\tau$ on a semigroup $S$ is called:
\begin{itemize}
  \item a \emph{semigroup} topology if $(S,\tau)$ is a topological semigroup;
  \item an \emph{inverse semigroup} topology if $(S,\tau)$ is a topological inverse semigroup;
  \item a \emph{shift-continuous} topology if $(S,\tau)$ is a semitopological semigroup;
  \item an \emph{inverse shift-continuous} topology if $(S,\tau)$ is a semitopological semigroup with continuous inversion.
\end{itemize}

The bicyclic monoid ${\mathscr{C}}(p,q)$ is the semigroup with the identity $1$ generated by two elements $p$ and $q$ subjected only to the condition $pq=1$. The semigroup operation on ${\mathscr{C}}(p,q)$ is determined as
follows:
\begin{equation*}
    q^kp^l\cdot q^mp^n=
    \left\{
      \begin{array}{ll}
        q^{k-l+m}p^n, & \hbox{if~} l<m;\\
        q^kp^n,       & \hbox{if~} l=m;\\
        q^kp^{l-m+n}, & \hbox{if~} l>m.
      \end{array}
    \right.
\end{equation*}
It is well known that the bicyclic monoid ${\mathscr{C}}(p,q)$ is a bisimple (and hence simple) combinatorial $E$-unitary inverse semigroup and every non-trivial congruence on ${\mathscr{C}}(p,q)$ is a group congruence \cite{Clifford-Preston=1961}.

\smallskip

It is well known that topological algebra studies the influence of
topological properties of its objects on their algebraic properties
and the influence of algebraic properties of its objects on their
topological properties. There are two main problems in topological
algebra: the problem of non-discrete topologization and the problem
of embedding into objects with some topological-algebraic
properties.

\smallskip

In mathematical literature the question about non-discrete
(Hausdorff) topologization was posed by Markov \cite{Markov=1945}.
Pontryagin gave well known conditions a base at the unity of a group
for its non-discrete topologization (see Theorem~4.5 of
\cite{Hewitt-Roos=1963} or Theorem~3.9 of \cite{Pontryagin=1966}).
Various authors have refined Markov's question: can a given infinite
group $G$ endowed with a non-discrete group topology be embedded
into a compact topological group? Again, for an arbitrary Abelian
group $G$ the answer is affirmative, but there is a non-Abelian
topological group that cannot be embedded into any compact
topological group ({see Section~9 of \cite{HBSTT-1984}}).

\smallskip

Also, Ol'shanskiy \cite{Olshansky=1980} constructed an infinite
countable group $G$ such that every Hausdorff group topology on $G$
is discrete. Taimanov presented in \cite{Taimanov=1973} a commutative semigroup $\mathfrak{T}$ which admits only discrete Hausdorff semigroup topology and gave in \cite{Taimanov=1975} sufficient conditions on a commutative semigroup to have a non-discrete semigroup topology. In \cite{Gutik=2016} it is proved that each $T_1$-topology with continuous shifts on $\mathfrak{T}$ is discrete.

\smallskip

The bicyclic monoid admits only the discrete semigroup Hausdorff topology \cite{Eberhart-Selden=1969}. Bertman and  West in \cite{Bertman-West=1976} extended this result for the case of Hausdorff semitopological semigroups. If a Hausdorff (semi)topological semigroup $T$ contains the bicyclic monoid ${\mathscr{C}}(p,q)$ as a dense proper semigroup then $T\setminus {\mathscr{C}}(p,q)$ is a closed ideal of $T$ \cite{Eberhart-Selden=1969, Gutik=2015}. Moreover, the closure of ${\mathscr{C}}(p,q)$ in a locally compact topological inverse semigroup can be obtained (up to isomorphism) from ${\mathscr{C}}(p,q)$ by adjoining the additive group of integers in a suitable way \cite{Eberhart-Selden=1969}.

\smallskip

Stable and $\Gamma$-compact topological semigroups do not contain the bicyclic monoid~\cite{Anderson-Hunter-Koch=1965, Hildebrant-Koch=1986, Koch-Wallace=1957}. The problem of embedding the bicyclic monoid into compact-like topological semigroups was studied in \cite{Banakh-Dimitrova-Gutik=2009, Banakh-Dimitrova-Gutik=2010, Bardyla-Ravsky=2020, Gutik-Repovs=2007}.

\smallskip

In this paper we construct two non-discrete inverse semigroup $T_1$-topologies and a compact inverse shift-continuous $T_1$-topologe on the bicyclic monoid ${\mathscr{C}}(p,q)$. Also we give  conditions on a $T_1$-topology $\tau$ on ${\mathscr{C}}(p,q)$ to be discrete. In particular, we show that if $\tau$ is an inverse semigroup $T_1$-topology on ${\mathscr{C}}(p,q)$ which satisfies one of the following conditions: $\tau$ is baire, $\tau$ is quasi-regular or $\tau$ is semiregular, then $\tau$ is discrete.

\smallskip

\section{\textbf{Examples of semigroup non-discrete $T_1$-topologies on the bicyclic monoid}}

In the following two examples we construct non-discrete $T_1$-semigroup inverse topologies on the bicyclic monoid.

\begin{example}\label{example-2.1}
We construct the topology $\tau_1$ on ${\mathscr{C}}(p,q)$ in the following way. For any $q^ip^j\in{\mathscr{C}}(p,q)$ and $n\in\omega$ we denote
\begin{equation*}
  U_n(q^ip^j)=\left\{q^ip^j\right\}\cup\left\{q^sp^t\colon s,t\geqslant n\right\}.
\end{equation*}
Let $\mathscr{B}_1(q^ip^j)=\left\{U_n(q^ip^j)\colon n\in\omega\right\}$ be the system of open neighbourhoods at the point $q^ip^j\in{\mathscr{C}}(p,q)$. It is obvious that the family $\mathscr{B}_1=\displaystyle\bigcup_{i,j\in\omega}\mathscr{B}_1(q^ip^j)$ satisfies the properties (BP1)--(BP3) of \cite{Engelking=1989}, and hence it generates a topology on ${\mathscr{C}}(p,q)$.
\end{example}

\begin{proposition}\label{proposition-2.2}
$\left({\mathscr{C}}(p,q),\tau_1\right)$ is a $T_1$-topological inverse semigroup.
\end{proposition}

\begin{proof}
It is obvious that $\tau_1$ is a $T_1$-topology on ${\mathscr{C}}(p,q)$.

\smallskip

Fix arbitrary $q^{i_1}p^{j_1},q^{i_2}p^{j_2}\in{\mathscr{C}}(p,q)$ and $n\in\omega$. The definition of the semigroup operation on the bicyclic semigroup ${\mathscr{C}}(p,q)$ and routine calculations imply that
\begin{equation*}
  U_m(q^{i_1}p^{j_1})\cdot U_m(q^{i_2}p^{j_2})\subseteq U_n(q^{i_1}p^{j_1}\cdot q^{i_2}p^{j_2})
\end{equation*}
and
\begin{equation*}
  \left(U_n(q^{i_1}p^{j_1})\right)^{-1}=\left(U_n(q^{j_1}p^{i_1})\right),
\end{equation*}
for $m=\max\left\{2n,i_1,j_1,i_2,j_2\right\}$. This completes the proof of the proposition.
\end{proof}

For the natural partial order $\preccurlyeq$ on the bicyclic semigroup ${\mathscr{C}}(p,q)$ and any $q^ip^j\in{\mathscr{C}}(p,q)$ we denote
\begin{align*}
  {\uparrow_{\preccurlyeq}}q^ip^j         &=\left\{q^sp^t\in{\mathscr{C}}(p,q)\colon q^ip^j\preccurlyeq q^sp^t\right\}; \\
  {\downarrow_{\preccurlyeq}}q^ip^j       &=\left\{q^sp^t\in{\mathscr{C}}(p,q)\colon q^sp^t\preccurlyeq q^ip^j\right\}; \\
  {\updownarrow_{\preccurlyeq}}q^ip^j     &={\uparrow_{\preccurlyeq}}q^ip^j\cup {\downarrow_{\preccurlyeq}}q^ip^j; \\
  {\downarrow_{\preccurlyeq}^\circ}q^ip^j &={\downarrow_{\preccurlyeq}}q^ip^j\setminus\left\{q^ip^j\right\}.
\end{align*}

The following statement describes the natural partial order $\preccurlyeq$ on the bicyclic semigroup ${\mathscr{C}}(p,q)$ and it follows from Lemma~1 of \cite{Gutik-Maksymyk=2016}.

\begin{lemma}\label{lemma-2.3}
Let $q^ip^j$ and $q^sp^t$ be arbitrary elements of the bicyclic semigroup ${\mathscr{C}}(p,q)$. Then the following statements are equivalent:
\begin{enumerate}
  \item[$(i)$]   $q^ip^j\preccurlyeq q^sp^t$;
  \item[$(ii)$]  $i\geqslant s$ and $i-j=s-t$;
  \item[$(iii)$] $j\geqslant t$ and $i-j=s-t$.
\end{enumerate}
\end{lemma}

The semigroup operation on the bicyclic semigroup ${\mathscr{C}}(p,q)$ and Lemma~\ref{lemma-2.3} imply the following lemma.

\begin{lemma}\label{lemma-2.4}
If $q^ip^j$ and $q^sp^t$ are arbitrary elements of the bicyclic semigroup ${\mathscr{C}}(p,q)$, then
\begin{equation*}
{\updownarrow_{\preccurlyeq}}q^ip^j\cdot {\updownarrow_{\preccurlyeq}}q^sp^t={\updownarrow_{\preccurlyeq}}q^{i+s}p^{j+t}.
\end{equation*}
\end{lemma}

\begin{example}\label{example-2.5}
We construct the topology $\tau_2$ on ${\mathscr{C}}(p,q)$ in the following way. For any $q^ip^j\in{\mathscr{C}}(p,q)$ and any non-negahtive integer $n$ we denote
\begin{equation*}
  O_n(q^ip^j)=\left\{q^ip^j\right\}\cup{\downarrow_{\preccurlyeq}^\circ}q^{i+n}p^{j+n}.
\end{equation*}
Let $\mathscr{B}_2(q^ip^j)=\left\{O_n(q^ip^j)\colon n\in\omega\right\}$ be the system of open neighbourhoods at the point $q^ip^j\in{\mathscr{C}}(p,q)$. It is obvious that the family $\mathscr{B}_2=\displaystyle\bigcup_{i,j\in\omega}\mathscr{B}_2(q^ip^j)$ satisfies the properties (BP1)--(BP3) of \cite{Engelking=1989}, and hence it generates a topology on ${\mathscr{C}}(p,q)$.
\end{example}

\begin{proposition}\label{proposition-2.6}
$\left({\mathscr{C}}(p,q),\tau_2\right)$ is a $T_1$-topological inverse locally compact semigroup.
\end{proposition}

\begin{proof}
It is obvious that $\tau_2$ is a $T_1$-topology on ${\mathscr{C}}(p,q)$. Also, simple verifications show that  for each $q^ip^j\in{\mathscr{C}}(p,q)$ and any open basic neighbourhood $O_n(q^ip^j)$ of $q^ip^j$ we have that the set ${\updownarrow_{\preccurlyeq}}q^ip^j\setminus O_n(q^ip^j)$ is finite and
\begin{equation*}
\operatorname{cl}_{\left({\mathscr{C}}(p,q),\tau_2\right)}(O_n(q^ip^j))={\updownarrow_{\preccurlyeq}}q^ip^j.
\end{equation*}
This implies that the space ${\updownarrow_{\preccurlyeq}}q^ip^j$ is compact and hence $\left({\mathscr{C}}(p,q),\tau_2\right)$ is locally compact.

\smallskip

Fix arbitrary $q^{i_1}p^{j_1},q^{i_2}p^{j_2}\in{\mathscr{C}}(p,q)$ and $n\in\omega$. The definition of the semigroup operation on the bicyclic semigroup ${\mathscr{C}}(p,q)$ and routine calculations imply that
\begin{equation*}
  O_m(q^{i_1}p^{j_1})\cdot O_m(q^{i_2}p^{j_2})\subseteq O_n(q^{i_1}p^{j_1}\cdot q^{i_2}p^{j_2})
\end{equation*}
and
\begin{equation*}
  \left(O_n(q^{i_1}p^{j_1})\right)^{-1}=\left(O_n(q^{j_1}p^{i_1})\right),
\end{equation*}
for $m=\max\left\{2n,i_1,j_1,i_2,j_2\right\}$, which completes the proof of the proposition.
\end{proof}

The following example shows that the bicyclic semigroup ${\mathscr{C}}(p,q)$ admits inverse shift-continuous compact $T_1$-topology.

\begin{example}\label{example-2.7}
We construct the topology $\tau_{\mathrm{c}}$ on ${\mathscr{C}}(p,q)$ in the following way. For any non-negahtive integer $n$ we denote
\begin{equation*}
  C_n=\left\{q^ip^j\in{\mathscr{C}}(p,q)\colon i,j\leqslant n\right\}.
\end{equation*}
Let
\begin{equation*}
\mathscr{B}_{\mathrm{c}}(q^ip^j)=\left\{W_n(q^ip^j)=\left\{q^ip^j\right\}\cup{\mathscr{C}}(p,q)\setminus C_n\colon n\in\omega\right\}
\end{equation*}
be the system of open neighbourhoods at the point $q^ip^j\in{\mathscr{C}}(p,q)$. It is obvious that the family $\mathscr{B}_{\mathrm{c}}=\displaystyle\bigcup_{i,j\in\omega}\mathscr{B}_{\mathrm{c}}(q^ip^j)$ satisfies the properties (BP1)--(BP3) of \cite{Engelking=1989}, and hence it generates the topology $\tau_{\mathrm{c}}$ on ${\mathscr{C}}(p,q)$.
\end{example}

\begin{proposition}\label{proposition-2.8}
$\tau_{\mathrm{c}}$ is an inverse shift-continuous compact  $T_1$-topology  on ${\mathscr{C}}(p,q)$.
\end{proposition}

\begin{proof}
It is obvious that $\tau_{\mathrm{c}}$ is a $T_1$-topology on ${\mathscr{C}}(p,q)$.  Since any basic open set is co-finite in $\left({\mathscr{C}}(p,q),\tau_{\mathrm{c}}\right)$, the space $\left({\mathscr{C}}(p,q),\tau_{\mathrm{c}}\right)$ is compact.

\smallskip

Since $\left(W_n(q^ip^j)\right)^{-1}=W_n(q^jp^i)$, the inversion is continuous in $\left({\mathscr{C}}(p,q),\tau_{\mathrm{c}}\right)$.

\smallskip

Fix arbitrary $ q^ip^j, q^kp^l\in {\mathscr{C}}(p,q)$.
Let $m\geqslant \max\left\{i,j,k,l\right\}$. By the definition of the semigroup operation in ${\mathscr{C}}(p,q)$ we get that  the following equalities hold
\begin{equation*}
  q^ip^j\cdot q^sp^t=
  \left\{
    \begin{array}{ll}
      q^{i-j+s}p^t, & \hbox{if~} 0\leqslant j\leqslant s\leqslant m \hbox{~and~} t >m;\\
      q^ip^{j-s+t}, & \hbox{if~} j\geqslant s, \, 0\leqslant s\leqslant m \hbox{~and~} t >m;\\
      q^{i-j+s}p^t, & \hbox{if~} s> m \hbox{~and~} t\in\omega
    \end{array}
  \right.
\end{equation*}
and
\begin{equation*}
 q^sp^t\cdot q^ip^j=
  \left\{
    \begin{array}{ll}
      q^sp^{t-i+j}, & \hbox{if~} 0\leqslant m \hbox{~and~} t >m;\\
      q^{s-t+i}p^j, & \hbox{if~} 0\leqslant t\leqslant i \hbox{~and~} s>m;\\
      q^sp^{t-i+j}, & \hbox{if~} t>i \hbox{~and~} s>m,
    \end{array}
  \right.
\end{equation*}
which imply that
\begin{equation*}
  q^ip^j\cdot W_{2m}(q^kp^l)\subseteq W_{m}(q^ip^j\cdot q^kp^l)
\end{equation*}
and
\begin{equation*}
  W_{2m}(q^kp^l)\cdot q^ip^j\subseteq W_{m}(q^kp^l\cdot q^ip^j),
\end{equation*}
respectively, and hence $\tau_{\mathrm{c}}$ is a shift-continuous  $T_1$-topology  on ${\mathscr{C}}(p,q)$.
\end{proof}

\section{\textbf{When a $T_1$-topology on the bicyclic monoid is discrete?}}

Next we shall study topological properties $\mathscr{P}$ such that if a $T_1$-topological space $({\mathscr{C}}(p,q),\tau)$ has property $\mathscr{P}$ and $\tau$ is a shift-continuous (semigroup, inverse semigroup) topology on ${\mathscr{C}}(p,q)$, then $\tau$ is discrete. The first such $\mathscr{P}$-property is the property to be a Baire space.

\smallskip

We recall that a topological space $X$ is said to be  \emph{Baire} if for each sequence $A_1, A_2,\ldots, A_i,\ldots$ of dense open subsets of $X$ the intersection $\displaystyle\bigcap_{i=1}^\infty A_i$ is a dense subset of $X$ \cite{Haworth-McCoy=1977}.

\begin{remark}\label{remark-2.7}
The topological space $\left({\mathscr{C}}(p,q),\tau_2\right)$ is not Baire, because $\left({\mathscr{C}}(p,q),\tau_2\right)$ has no an isolated point in itself (see Proposition~1.30 in \cite{Haworth-McCoy=1977}). But $\left({\mathscr{C}}(p,q),\tau_2\right)$ is a locally compact space. Indeed, the set ${\updownarrow_{\preccurlyeq}}q^ip^j$ is compact for any $q^ip^j\in {\mathscr{C}}(p,q)$, because the set ${\updownarrow_{\preccurlyeq}}q^ip^j\setminus O_n(q^ip^j)$ is finite for all $O_n(q^ip^j)\in \mathscr{B}_2(q^ip^j)$. Moreover, for any $O_n(q^ip^j)\in \mathscr{B}_2(q^ip^j)$ we have that $\operatorname{cl}_{\left({\mathscr{C}}(p,q),\tau_2\right)}(O_n(q^ip^j))={\updownarrow_{\preccurlyeq}}q^ip^j$.
\end{remark}

\begin{theorem}\label{theorem--2.8}
Every shift-continuous Baire $T_1$-topology $\tau$ on the bicyclic monoid ${\mathscr{C}}(p,q)$ is discrete.
\end{theorem}

\begin{proof}
By Proposition 1.30 of \cite{Haworth-McCoy=1977} the space $({\mathscr{C}}(p,q),\tau)$ has an isolated point $q^ip^j$. Then for an arbitrary point $q^mp^n$ in $({\mathscr{C}}(p,q),\tau)$ the separate continuity of the semigroup operation in $({\mathscr{C}}(p,q),\tau)$ implies that there exists an open neighbourhood $U(q^mp^n)$ of $q^mp^n$ in $({\mathscr{C}}(p,q),\tau)$ such that
\begin{equation*}
  q^ip^m\cdot U(q^mp^n)\cdot q^np^j\subseteq \left\{q^ip^j\right\}.
\end{equation*}
By Lemma~I.1 of \cite{Eberhart-Selden=1969} the equations $A\cdot X=B$ and $X\cdot C=D$ have only finite sets of solutions in ${\mathscr{C}}(p,q)$, and hence the set $U(q^mp^n)$ is finite. Since $\tau$ is a $T_1$-topology, the point $q^mp^n$ is isolated in $({\mathscr{C}}(p,q),\tau)$. This completes the proof of the theorem.
\end{proof}

\begin{lemma}\label{lemma-2.9}
Let $\tau$ be a shift-continuous $T_1$-topology on the bicyclic monoid ${\mathscr{C}}(p,q)$ such that the maps ${\mathscr{C}}(p,q)\to E({\mathscr{C}}(p,q))$, $x\mapsto xx^{-1}$ and ${\mathscr{C}}(p,q)\to E({\mathscr{C}}(p,q))$, $x\mapsto x^{-1}x$ are continuous. If for some idempotent $q^ip^i\in {\mathscr{C}}(p,q)$, $i\in\omega$, there exists an open neighbourhood $U(q^ip^i)$ of  $q^ip^i$ in $({\mathscr{C}}(p,q),\tau)$ such that the set $U(q^ip^i)\cap E({\mathscr{C}}(p,q))$ is finite, then $\tau$ is discrete.
\end{lemma}

\begin{proof}
Since $\tau$ is a $T_1$-topology on ${\mathscr{C}}(p,q)$, without loss of generality we may assume that $U(q^ip^i)\cap E({\mathscr{C}}(p,q))=\left\{q^ip^i\right\}$. By Lemma~I.1 of \cite{Eberhart-Selden=1969} the equations $A\cdot X=B$ and $X\cdot C=D$ have only finite sets of solutions in ${\mathscr{C}}(p,q)$, and hence the separate continuity of the semigroup operation in $({\mathscr{C}}(p,q),\tau)$ implies that for any idempotent $q^jp^j\in {\mathscr{C}}(p,q)$, $j\in\omega$, there exists an open neighbourhood $V(q^jp^j)$ of  $q^ip^i$ in $({\mathscr{C}}(p,q),\tau)$ such that
\begin{equation*}
  q^ip^j\cdot V(q^jp^j)\cdot q^jp^i\subseteq U(q^ip^i).
\end{equation*}
Also, by the definition of the semigroup operation on ${\mathscr{C}}(p,q)$ we get that the set $U(q^jp^j)\cap E({\mathscr{C}}(p,q))$ is finite, as well. Hence without loss of generality we may assume that every idempotent $q^jp^j\in {\mathscr{C}}(p,q)$, $j\in\omega$ has an open neighbourhood $W(q^jp^j)$ in $({\mathscr{C}}(p,q),\tau)$ such that $W(q^jp^j)\cap E({\mathscr{C}}(p,q))=\left\{q^jp^j\right\}$.

Since the maps ${\mathscr{C}}(p,q)\to E({\mathscr{C}}(p,q))$, $x\mapsto xx^{-1}$ and ${\mathscr{C}}(p,q)\to E({\mathscr{C}}(p,q))$, $x\mapsto x^{-1}x$ are continuous, for any point $q^mp^n\in {\mathscr{C}}(p,q)$, $m.n\in\omega$, there exists an open neighbourhood $O(q^mp^n)$ of the point $q^mp^n$ in $({\mathscr{C}}(p,q),\tau)$ such that
\begin{equation*}
q^{m_1}p^{n_1}\cdot (q^{m_1}p^{n_1})^{-1}=q^{m_1}p^{m_1}\subseteq \{q^{m}p^{m}\}
\end{equation*}
and
\begin{equation*}
  (q^{m_1}p^{n_1})^{-1}\cdot q^{m_1}p^{n_1}=q^{n_1}p^{n_1}\subseteq \{q^{n}p^{n}\},
\end{equation*}
for all $q^{m_1}p^{n_1}\in O(q^mp^n)$. The last two inclusions imply that the neighbourhood $O(q^mp^n)$ is a singleton, i.e., $O(q^mp^n)=\left\{q^mp^n\right\}$. This implies the statement of the lemma.
\end{proof}

Let $X$ be a topological space and $Y$ be a subspace of $X$. We shall say that the space \emph{$Y$ is quasi-regular at a point} $x\in Y$ if for any open neighbourhood $U(x)$ of $x$ in $Y$ there exists an open nonempty subset $V$ in $Y$ such that $\operatorname{cl}_Y(V)\subseteq U(x)$.

\begin{lemma}\label{lemma-2.10}
Let $\tau$ be a shift-continuous $T_1$-topology on ${\mathscr{C}}(p,q)$. If there exists a point $q^ip^j\in {\mathscr{C}}(p,q)$ such that ${\updownarrow_{\preccurlyeq}}q^ip^j$ is quasi-regular at $q^ip^j$, then for any point $q^mp^n\in {\mathscr{C}}(p,q)$ the space ${\updownarrow_{\preccurlyeq}}q^mp^n$ is quasi-regular at  $q^mp^n$.
\end{lemma}

\begin{proof}
First we observe that for any $q^ip^j\in {\mathscr{C}}(p,q)$ the set ${\downarrow_{\preccurlyeq}}q^ip^j$ is open in ${\updownarrow_{\preccurlyeq}}q^ip^j$ because the $\tau$ is a $T_1$-topology on ${\mathscr{C}}(p,q)$.

\smallskip

We define the mapping $\mathfrak{f}_{q^ip^j}^{q^mp^n}\colon {\mathscr{C}}(p,q)\to {\mathscr{C}}(p,q)$ by the formula $\mathfrak{f}_{q^ip^j}^{q^mp^n}(x)=q^ip^m\cdot x\cdot q^np^j$, for any $i,j,m,n\in \omega$. Then by Lemma~\ref{lemma-2.3} we have that $q^{m+k}p^{n+k}\in {\downarrow_{\preccurlyeq}}q^ip^j$ and the semigroup operation in ${\mathscr{C}}(p,q)$ implies that
\begin{align*}
  \mathfrak{f}_{q^ip^j}^{q^mp^n}(q^{m+k}p^{n+k})&=q^ip^m\cdot q^{m+k}p^{n+k}\cdot q^np^j= \\
   &=q^i(p^mq^{m+k})(p^{n+k}q^n)p^j=\\
   &=q^iq^kp^kp^j=\\
   &=q^{i+k}p^{j+k},
\end{align*}
for any $k\in\omega$. Hence, the restrictions
\begin{equation*}
  \mathfrak{f}_{q^ip^j}^{q^mp^n}{\upharpoonleft}_{{\downarrow_{\preccurlyeq}}q^mp^n}\colon {\downarrow_{\preccurlyeq}}q^mp^n\to {\downarrow_{\preccurlyeq}}q^ip^j \qquad \hbox{and} \qquad \mathfrak{f}_{q^mp^n}^{q^ip^j}{\upharpoonleft}_{{\downarrow_{\preccurlyeq}}q^ip^j}\colon {\downarrow_{\preccurlyeq}}q^ip^j\to {\downarrow_{\preccurlyeq}}q^mp^n
\end{equation*}
are mutually inverse mappings and by the separate continuity of the semigroup operation in $({\mathscr{C}}(p,q),\tau)$ they are homeomorphisms. Since the set ${\downarrow_{\preccurlyeq}}q^sp^t$ is open in ${\updownarrow_{\preccurlyeq}}q^sp^t$ for any $s,t\in \omega$, the above arguments imply the statement of the lemma.
\end{proof}

\begin{proposition}\label{proposition-2.11}
Let $\tau$ be an inverse semigroup $T_1$-topology on ${\mathscr{C}}(p,q)$. If there exists an idempotent $q^ip^i\in {\mathscr{C}}(p,q)$ such that the space $E({\mathscr{C}}(p,q))$  is quasi-regular at $q^ip^i$, then $\tau$ is discrete.
\end{proposition}

\begin{proof}
Let $U(q^ip^i)$ be an open neighbourhood of the point $q^ip^i$ in $E({\mathscr{C}}(p,q))$. Without loss of generality we may assume that the set $U(q^ip^i)$ is infinite, because otherwise by Lemma \ref{lemma-2.9} the topological space $({\mathscr{C}}(p,q),\tau)$ is discrete. Since $({\mathscr{C}}(p,q),\tau)$ is a $T_1$-space, $V_{q^ip^i}=U(q^ip^i)\setminus \{q^ip^i\}$ is an open set in $E({\mathscr{C}}(p,q))$. Then there exists a nonempty open subset $W_{q^ip^i}\subseteq V_{q^ip^i}$ such that $\operatorname{cl}_{E({\mathscr{C}}(p,q))}(W_{q^ip^i})\subseteq V_{q^ip^i}$. Hence
\begin{equation*}
O(q^ip^i)=U(q^ip^i)\setminus \operatorname{cl}_{E({\mathscr{C}}(p,q))}(W_{q^ip^i})
\end{equation*}
is an open neighbourhood of the point $q^ip^i$ in $E({\mathscr{C}}(p,q))$. Without loss of generality we may assume that the set $W_{q^ip^i}$ is infinite, because otherwise there exists an idempotent in $({\mathscr{C}}(p,q),\tau)$ which has a finite open neighbourhood, and hence by Lemma \ref{lemma-2.9} the topological space $({\mathscr{C}}(p,q),\tau)$ is discrete. The structure of the natural partial order $\preccurlyeq$ on the bicyclic monoid ${\mathscr{C}}(p,q)$ implies that the set ${\uparrow_{\preccurlyeq}}q^ip^i$ is finite, and hence there exists an idempotent $q^jp^j\in W_{q^ip^i}$ such that $q^jp^j\in {\downarrow_{\preccurlyeq}^\circ}q^ip^i$. Then $q^jp^j\cdot q^ip^i=q^jp^j$ and the continuity of the semigroup operation in $({\mathscr{C}}(p,q),\tau)$ implies that there exist open neighbourhoods $W_1(q^ip^i)$ and $W_1(q^jp^j)$ of the points $q^ip^i$ and $q^jp^j$ in $({\mathscr{C}}(p,q),\tau)$, respectively, such that
\begin{equation}\label{eq-2.1}
  (W_1(q^jp^j)\cap E({\mathscr{C}}(p,q)))\cdot (W_1(q^ip^i)\cap E({\mathscr{C}}(p,q)))\subseteq W_{q^ip^i},
\end{equation}
\begin{equation*}
  W_1(q^ip^i)\cap E({\mathscr{C}}(p,q))\subseteq O(q^ip^i),
\end{equation*}
\begin{equation*}
  W_1(q^jp^j)\cap E({\mathscr{C}}(p,q))\subseteq W_{q^ip^i},
\end{equation*}
and the sets $W_1(q^ip^i)\cap E({\mathscr{C}}(p,q))$ and $W_1(q^jp^j)\cap E({\mathscr{C}}(p,q))$ are infinite. The last two properties imply that for any
\begin{equation*}
q^kp^k\in W_1(q^jp^j)\cap E({\mathscr{C}}(p,q))
\end{equation*}
there exists
\begin{equation*}
q^lp^l\in W_1(q^ip^i)\cap E({\mathscr{C}}(p,q))
\end{equation*}
such that
\begin{equation*}
q^kp^k\cdot q^lp^l=q^lp^l\cdot q^kp^k=q^lp^l,
\end{equation*}
which contradicts condition \eqref{eq-2.1}. The obtained contradiction implies that at least one of the sets $W_1(q^ip^i)\cap E({\mathscr{C}}(p,q))$ or $W_1(q^jp^j)\cap E({\mathscr{C}}(p,q))$ is finite. Then by Lemma~\ref{lemma-2.9} the topology $\tau$ is discrete.
\end{proof}

Lemma~\ref{lemma-2.10} and Proposition~\ref{proposition-2.11} imply the following theorem.

\begin{theorem}\label{theorem-2.12}
Let $\tau$ be an inverse semigroup $T_1$-topology on ${\mathscr{C}}(p,q)$. If there exists a point $q^ip^j\in {\mathscr{C}}(p,q)$ such that the space ${\updownarrow_{\preccurlyeq}}q^ip^j$ is  quasi-regular at $q^ip^j$, then $\tau$ is discrete.
\end{theorem}

Let $X$ be a topological space and $Y$ be a subspace of $X$. We shall say that the space \emph{$Y$ is semiregular at point} $x\in Y$ if there exists a basis $\mathscr{B}(x)$ of the topology of the space $Y$ at $x$ which consists of regular open subsets of $Y$, i.e., $U=\operatorname{int}_Y(\operatorname{cl}_Y(U))$ for any $U\in \mathscr{B}(x)$.

\smallskip

The proof of the following lemma is similar to Lemma~\ref{lemma-2.10}.

\begin{lemma}\label{lemma-2.13}
Let $\tau$ be a shift-continuous $T_1$-topology on ${\mathscr{C}}(p,q)$. If there exists a point $q^ip^j\in {\mathscr{C}}(p,q)$ such that the space ${\updownarrow_{\preccurlyeq}}q^ip^j$ is semiregular at $q^ip^j$, then for any point $q^mp^n\in {\mathscr{C}}(p,q)$ the space ${\updownarrow_{\preccurlyeq}}q^mp^n$ is semiregular at  $q^mp^n$.
\end{lemma}

\begin{proposition}\label{proposition-2.14}
Let $\tau$ be a shift-continuous $T_1$-topology on the bicyclic monoid ${\mathscr{C}}(p,q)$ such that the maps ${\mathscr{C}}(p,q)\to E({\mathscr{C}}(p,q))$, $x\mapsto xx^{-1}$ and ${\mathscr{C}}(p,q)\to E({\mathscr{C}}(p,q))$, $x\mapsto x^{-1}x$ are continuous. If there exists an idempotent $q^ip^i\in {\mathscr{C}}(p,q)$ such that the space $E({\mathscr{C}}(p,q))$  is semiregular at $q^ip^i$, then $\tau$ is discrete.
\end{proposition}

\begin{proof}[\textsl{Proof}]
Suppose to the contrary that there exists an inverse semigroup non-discrete $T_1$-topology on ${\mathscr{C}}(p,q)$ such that he space $E({\mathscr{C}}(p,q))$  is semiregular at $q^ip^i$ for some idempotent $q^ip^i\in {\mathscr{C}}(p,q)$. We claim that $\operatorname{cl}_{E({\mathscr{C}}(p,q))}(U(q^ip^i))={\updownarrow_{\preccurlyeq}}q^ip^i$ for any regular open neighbourhood $U(q^ip^i)$ in $E({\mathscr{C}}(p,q))$ of the point $q^ip^i$.

\smallskip

Suppose to the contrary that there exists an idempotent $q^jp^j\in {\mathscr{C}}(p,q)$ such that
\begin{equation*}
q^jp^j\notin\operatorname{cl}_{E({\mathscr{C}}(p,q))}(U(q^ip^i)),
\end{equation*}
i.e., there exists an open neighbourhood $U(q^jp^j)$ of the point $q^jp^j$ in $E({\mathscr{C}}(p,q))$ such that $U(q^jp^j)\cap U(q^ip^i)=\varnothing$.
If the point  $q^jp^j$ has a finite neighbourhood, then by Lemma~\ref{lemma-2.9} the topology $\tau$ is discrete. Hence all open neighbourhoods of the point $q^jp^j$ are infinite in $E({\mathscr{C}}(p,q))$. If $j<i$ then $q^ip^i\cdot q^jp^j=q^ip^i$. The separate continuity of the semigroup operation in $\left({\mathscr{C}}(p,q),\tau\right)$ implies that for a regular open neighbourhood $U(q^ip^i)$ of $q^ip^i$ in $E({\mathscr{C}}(p,q))$ there exists an open neighbourhood $V(q^jp^j)\subseteq U(q^jp^j)$ of $q^jp^j$ in $E({\mathscr{C}}(p,q))$ such that
\begin{equation*}
  V(q^jp^j)\cdot q^ip^i\subseteq U(q^ip^i).
\end{equation*}
By the definition of the bicyclic semigroup ${\mathscr{C}}(p,q)$ the neighbourhood $V(q^jp^j)$ contains infinitely many idempotents $q^kp^k$, $k\in\omega$, such that $q^ip^i\cdot q^kp^k=q^kp^k$. Since $V(q^jp^j)\cap U(q^ip^i)=\varnothing$, this contradicts the inclusion
$
  V(q^jp^j)\cdot q^ip^i\subseteq U(q^ip^i).
$
If $j>i$ then $q^ip^i\cdot q^jp^j=q^jp^j$. The separate continuity of the semigroup operation in $\left({\mathscr{C}}(p,q),\tau\right)$ implies that for an open neighbourhood $U(q^jp^j)$ of $q^jp^j$ in $E({\mathscr{C}}(p,q))$ there exists a regular open neighbourhood $V(q^ip^i)\subseteq U(q^ip^i)$ of $q^ip^i$ in $E({\mathscr{C}}(p,q))$ such that $V(q^ip^i)\cdot q^jp^j\subseteq U(q^jp^j)$. Again, by the definition of the bicyclic semigroup ${\mathscr{C}}(p,q)$ the neighbourhood $V(q^ip^i)$ contains infinitely many idempotents $q^kp^k$, $k\in\omega$, such that $q^jp^j\cdot q^kp^k=q^kp^k$. Similar as in previous case we obtain a contradiction.

\smallskip

The obtained contradictions imply that
\begin{equation*}
\operatorname{cl}_{E({\mathscr{C}}(p,q))}(U(q^ip^i))={\updownarrow_{\preccurlyeq}}q^ip^i
\end{equation*}
for any regular open neighbourhood $U(q^ip^i)$ in $E({\mathscr{C}}(p,q))$ of the point $q^ip^i$. This equality contradicts the assumption that $\left({\mathscr{C}}(p,q),\tau\right)$ is a $T_1$-space. Hence $\tau$ is the discrete topology on the bicyclic monoid ${\mathscr{C}}(p,q)$.
\end{proof}

Lemma~\ref{lemma-2.13} and Proposition~\ref{proposition-2.14} imply the following theorem.

\begin{theorem}\label{theorem-2.15}
Let $\tau$ be a shift-continuous $T_1$-topology on the bicyclic monoid ${\mathscr{C}}(p,q)$ such that the maps ${\mathscr{C}}(p,q)\to E({\mathscr{C}}(p,q))$, $x\mapsto xx^{-1}$ and ${\mathscr{C}}(p,q)\to E({\mathscr{C}}(p,q))$, $x\mapsto x^{-1}x$ are continuous. If there exists a point $q^ip^j\in {\mathscr{C}}(p,q)$ such that the space ${\updownarrow_{\preccurlyeq}}q^ip^j$ is  semiregular  at $q^ip^j$, then $\tau$ is discrete.
\end{theorem}


\section*{\textbf{Acknowledgements}}

The authors acknowledge Taras Banakh and the referee for his/her comments and suggestions.

\end{document}